\documentclass{article}
\usepackage{amsmath}
\usepackage{amssymb}
\usepackage{amsfonts}

\newtheorem{theorem}{Theorem}[section]

\newtheorem{definition}[theorem]{Definition}

\newtheorem{notation}[theorem]{Notation}

\newenvironment{proof}[1][Proof]{\noindent \textbf{#1.} }{\  \rule{0.5em}{0.5em}}
\input{tcilatex}
\begin{document}

\title{Stochastic differential equations with path-independent solutions. }
\author{I.P. van den Berg \\
Departamento de Matem\'{a}tica, Universidade de \'{E}vora, Portugal\\
e-mail ivdb@uevora.pt}
\date{}
\maketitle

\begin{abstract}
We present a condition for a stochastic differential equation $dX_{t}=\mu
\left( t,X_{t}\right) dt+\sigma \left( t,X_{t}\right) dB_{t}$\ to have a
unique functional solution of the form $Z(t,B_{t})$. The condition expresses
a relation between $\mu $ and $\sigma $. A generalization concerns solutions
of the form $Z(t,Y_{t})$, where $Y_{t}$ is an Ito-process satisfying a
stochastic differential equation with coefficients only depending on time,
to be determined from $\mu $ and $\sigma $. The solutions in question are
obtained by solving a system of two partial differential equations, which
may be reduced to two ordinary differential equations.

\medskip


\medskip

\textit{AMS classfication}: 60H10, 35F20.
\end{abstract}

\section{Introduction}

Ito's lemma gives conditions for a stochastic differential equation to have
a solution in terms of a function $Z(t,Y_{t})$ of time and some Ito-process $%
Y_{t}$: the coefficients of the stochastic differential equation should have
a particular expression in terms of the partial derivatives of $Z$. Ito's
lemma does not provide a direct method to find the function $Z$.

For a definite class of stochastic differential equations we show that such
a function may be obtained in the form of a solution of two subsequent
ordinary differential equations. These equations may be reduced to some
equations used in the differential method of H. Doss in \cite{Doss}. This
method solves autonomous stochastic differential equations path-by-path,
formulating a particular ordinary differential equation along each
individual path of the process $Y$.

Here we give "integration conditions" on the coefficients of a given
non-autonomous stochastic differential equation, for it to be solved by a
global function $Z$, through only one pair of ordinary differential
equations. The integration conditions permit to determine an auxiliary
process $Y$, and the function $Z$ will depend only on time $t$ and the
values taken by the process $Y$ at the time $t$; as such it is independent
of the paths of the process $Y$. We consider in particular the special case
where $Y$ may be taken equal to the Standard Brownian Motion.

The ordinary differential equations in question correspond to a system of
two first-order partial differential equations in three variables, solved
along two particular paths (in time and in space). In fact we show that the
above mentioned integration conditions correspond to a well-known
integration condition for systems of partial differential equations to have
a unique global solution.

This article has the following structure. In Section \ref{Sektiesamenvatting}
we present our differential approach in more detail, and compare it with the
differential method by H. Doss. In Section \ref{sectieexistentie} we recall
some existing theory on the resolution of systems of first-order partial
differential equations. In Section \ref{sectiestochastiek} we state formal
theorems on the solution of stochastic differential equations with their
respective proofs. We comment on the role and form of the integration
conditions, and end with some examples.

Thorough treatments of stochastic differential equations can be found in for
example \cite{Arnold}, \cite{Oksendahl} and \cite{Protter}. The latter book
also gives a presentation of the result of \cite{Doss}.

The books \cite{CourantHilbert} and \cite{Zwillinger} are books of reference
for partial differential equations, and a treatment of the background in
differential geometry useful for the solution of systems of first-order
partial differential equations is given in \cite{Boothby}.

The integration condition of a global solution of stochastic differential
equations in terms of Brownian Motion has been stated as a sort of limit
case for the existence of a global solution of stochastic difference
equations in terms of the discrete Wiener Walk in \cite{vandenberglobry}.

\section{Overview of the method\label{Sektiesamenvatting}}

We consider stochastic differential equations of the form 
\begin{equation}
\left \{ 
\begin{array}{lllll}
dX_{t} & = & \mu \left( t,X_{t}\right) dt+\sigma \left( t,X_{t}\right) dB_{t}
& \qquad & 0\leq t<T \\ 
X_{0} & = & x_{0}, &  & 
\end{array}%
\right.  \label{stochdifverg}
\end{equation}%
on some appropriate probability space $\Omega $, where $\mu $ and $\sigma $
have some regularity and $x_{0}$ is a constant.

We show that (\ref{stochdifverg}) has a functional solution of the form $%
Z(t,B_{t})$, where $Z:[0,T]\times \mathbb{R}\rightarrow \mathbb{R}$,
provided $\mu $ and $\sigma $ are related by the partial differential
equation 
\begin{equation}
\sigma \frac{\partial \mu }{\partial X}-\mu \frac{\partial \sigma }{\partial
X}-\frac{\partial \sigma }{\partial t}-\frac{\sigma ^{2}}{2}\frac{\partial
^{2}\sigma }{\partial X^{2}}=0.  \label{integratieconditiemsintroductie}
\end{equation}%
In fact $Z$ satisfies the system of two first-order partial differential
equations 
\begin{equation}
\left \{ 
\begin{array}{lll}
\frac{\partial Z(t,x)}{\partial t} & = & \mu (t,Z(t,x))-\frac{1}{2}\sigma
(t,Z(t,x))\frac{\partial \sigma (t,Z(t,x))}{\partial Z} \\ 
\frac{\partial Z(t,x)}{\partial x} & = & \sigma (t,Z(t,x)) \\ 
Z(0,0) & = & x_{0},%
\end{array}%
\right.  \label{systeemz}
\end{equation}%
and (\ref{integratieconditiemsintroductie}) represents the integration
condition for the system (\ref{systeemz}). Indeed, for a well-defined
two-times differentiable surface to be a solution, one must have $\frac{%
\partial ^{2}Z}{\partial t\partial x}=\frac{\partial ^{2}Z}{\partial
x\partial t}$, which amounts to (\ref{integratieconditiemsintroductie}).
Under an additional Lipschitz condition the solution of (\ref{systeemz}) is
also unique, meaning that it does not depend on the choice of the path of
integration. Hence any convenient path will do. In particular one may
integrate firstly horizontally from $(0,0)$ to $(T,0)$, and then vertically
from $(T,0)$ to $(T,x)$, thus resolving subsequently the two ordinary
differential equations%
\begin{equation}
\left \{ 
\begin{array}{lll}
\frac{dZ}{dt} & = & \mu (t,Z)-\frac{1}{2}\sigma (t,Z)\frac{\partial \sigma
(t,Z)}{\partial Z} \\ 
Z(0) & = & x_{0}%
\end{array}%
\right.  \label{Vergelijkinggewoon1}
\end{equation}%
and%
\begin{equation}
\left \{ 
\begin{array}{lll}
\frac{dZ_{T}}{dx} & = & \sigma (T,Z_{T}) \\ 
Z_{T}(0) & = & Z_{0}(T),%
\end{array}%
\right.  \label{Vergelijkinggewoon2}
\end{equation}%
with $Z(T,x)=Z_{T}(x)$. In principle, the solution of the successive
ordinary differential equations (\ref{Vergelijkinggewoon1}) and (\ref%
{Vergelijkinggewoon2}) gives an a priori method to solve a class of
stochastic differential equations, in contrast to the usual a posteriori
justification by Ito's Lemma, where one concludes that a given function
satisfies the stochastic differential equation by verifying some identities
in terms of the coefficients of the equation and some partial derivatives of
the function. As such, the method is a sort of a converse to Ito's Lemma.

The method of solution by (\ref{Vergelijkinggewoon1}) and (\ref%
{Vergelijkinggewoon2}) is similar to the differential representation of the
solution of autonomous stochastic differential equations 
\begin{equation}
\left \{ 
\begin{array}{lllll}
dX_{t} & = & \mu \left( X_{t}\right) dt+\sigma \left( X_{t}\right) dB_{t} & 
\qquad & 0\leq t<T \\ 
X_{0} & = & x_{0}, &  & 
\end{array}%
\right.  \label{stochdifvergautonoom}
\end{equation}%
of H. Doss in \cite{Doss} (see also \cite{Protter}), which is also an a
priori method based on the successive integration of two ordinary
differential equations.

The representation in \cite{Doss} of the solution of (\ref{stochdifverg}) is%
\begin{equation*}
X_{t}=H(D_{t},B_{t}),
\end{equation*}%
where $D_{t}(\omega )$ satisfies for nearly all $\omega \in \Omega $%
\begin{equation}
\left \{ 
\begin{array}{l}
\begin{array}{l}
D_{t}^{\prime }(\omega )=\exp \left( -\int \limits_{0}^{B_{t}(\omega )}\sigma
^{\prime }(H(D_{t}(\omega ),\xi )d\xi \right) \times \\ 
\text{ \  \  \  \  \  \ }\left( \mu (H(D_{t}(\omega ),B_{t}(\omega )))-\frac{1}{2}%
\sigma (H(D_{t}(\omega ),B_{t}(\omega )))\sigma ^{\prime }(H(D_{t}(\omega
),B_{t}(\omega )))\right)%
\end{array}
\\ 
D_{0}(\omega )=x_{0},%
\end{array}%
\right.  \label{VergelijkingHD1}
\end{equation}%
and $H$ satisfies%
\begin{equation}
\left \{ 
\begin{array}{lll}
\frac{\partial H(D,B)}{\partial B} & = & \sigma (H(D,B)) \\ 
H(D,0) & = & D.%
\end{array}%
\right.  \label{VergelijkingHD2}
\end{equation}

To compare both integrations, let us begin by noting that in the latter
approach typically firstly the equation (\ref{VergelijkingHD2}) is solved
and then the equation in time (\ref{VergelijkingHD1}), while the former
method starts with the equation in time (\ref{Vergelijkinggewoon1}) followed
by the equation in space (\ref{Vergelijkinggewoon2}).

The equation in time (\ref{Vergelijkinggewoon1}) corresponds to integrating (%
\ref{VergelijkingHD1}) along the path, say $\omega _{0}$, of Brownian Motion
which is everywhere $0$; then the integration of (\ref{VergelijkingHD2}) is
only over an interval of length $0$. Indeed, identifying $D_{t}(\omega _{0})$
with a real function $D$, equation (\ref{VergelijkingHD1}) amounts to%
\begin{equation}
\left \{ 
\begin{array}{lll}
\frac{dD}{dt} & = & \mu (H(D,0))-\frac{1}{2}\sigma (H(D,0))\frac{d\sigma
(H(D,0))}{dH}=\mu (D)-\frac{1}{2}\sigma (D)\frac{d\sigma (D)}{dH} \\ 
D(0) & = & x_{0}.%
\end{array}%
\right.  \label{VergelijkingHDspeciaal1}
\end{equation}%
On the other hand (\ref{VergelijkingHD2}) becomes, identifying $H(D,B)$ with
a real function $H_{D}(B)$ at the point $D=D$,%
\begin{equation}
\left \{ 
\begin{array}{lll}
\frac{dH_{D}(B)}{dB} & = & \sigma (H_{D}(B))) \\ 
H_{D}(0) & = & D.%
\end{array}%
\right.  \label{VergelijkingHDspeciaal2}
\end{equation}%
From the initial condition of (\ref{VergelijkingHDspeciaal2}), one obtains%
\begin{equation*}
X_{t}(\omega _{0})=H(D_{t}(\omega _{0}),0)=D_{t}(\omega _{0}),
\end{equation*}%
as (trivial) solution of (\ref{VergelijkingHDspeciaal2}).

Suppose now that the stochastic differential equation (\ref{stochdifverg})
has indeed a solution $\widetilde{H}(t,B_{t})$ which depends only on time
and the values taken by Brownian motion. We will again recognize (\ref%
{Vergelijkinggewoon1}) when integrating along a horizontal path, say $\omega 
$, and the "vertical equation" (\ref{Vergelijkinggewoon2}). One notes \cite%
{Doss} that%
\begin{equation*}
\frac{\partial H(D,B)}{\partial D}=\exp \left( \int \limits_{0}^{B_{t}(\omega
)}\sigma ^{\prime }(H(D_{t}(\omega ),\xi )d\xi \right) .
\end{equation*}%
Hence we derive from (\ref{VergelijkingHD1}) indeed%
\begin{equation*}
\frac{\partial \widetilde{H}(t,B_{t})}{\partial t}=\text{\ }\mu (\widetilde{H%
}(t,B_{t}))-\frac{1}{2}\sigma (\widetilde{H}(t,B_{t}))\sigma ^{\prime }(%
\widetilde{H}(t,B_{t})),
\end{equation*}%
with $\widetilde{H}(t,B_{t})=H(D_{t}(\omega ),B_{t}(\omega ))$ given by 
\begin{equation*}
\left \{ 
\begin{array}{lll}
\frac{\partial \widetilde{H}(t,B_{t})}{\partial B} & = & \sigma (\widetilde{H%
}(t,B_{t})) \\ 
\widetilde{H}(t,0) & = & D_{t}(\omega ).%
\end{array}%
\right.
\end{equation*}

We recall that (\ref{stochdifverg}) has a global solution of type $%
\widetilde{H}(t,B_{t})$, only if $\mu $ and $\sigma $ satisfy the
integration condition (\ref{integratieconditiemsintroductie}).

The equation (\ref{VergelijkingHD1}) shows that the approach of \cite{Doss}
is essentially path-dependent: each individual path of Brownian Motion
generates a pair of ordinary differential equations which determines a
particular solution. On the contrary, if the integration condition (\ref%
{integratieconditiemsintroductie}) holds, one couple of ordinary
differential equations will yield a global solution $Z$, valid for all
paths. The solution is path-independent in the sense that if $\omega ,\omega
^{\prime }\in \Omega $ are such that $B_{t}(\omega )=B_{t}(\omega ^{\prime
}) $ at some time $t$, it holds that $Z(t,B_{t}(\omega ))=Z(t,B_{t}(\omega
^{\prime }))$.

In fact we will present a somewhat more general method to find a global
solution of stochastic differential equations, for a class of equations
which not necessarily satisfies (\ref{integratieconditiemsintroductie}).
Indeed, if%
\begin{equation}
\frac{\partial \mu }{\partial X}-\frac{\mu }{\sigma }\frac{\partial \sigma }{%
\partial X}-\frac{1}{\sigma }\frac{\partial \sigma }{\partial t}-\frac{%
\sigma }{2}\frac{\partial ^{2}\sigma }{\partial ^{2}X}=\phi (t)
\label{integratieconditieFGintroduktie}
\end{equation}%
for some real function $\phi $, a global solution of (\ref{stochdifverg})
can be found in terms of a deterministic function $Z(t,Y_{t})$ of time $t$
and an Ito-process $Y_{t}$, given by a stochastic integral of the form%
\begin{equation}
Y(T,\omega )=y_{0}+\int_{0}^{T}F\left( t\right) dt+\int_{0}^{T}G\left(
t\right) dB_{t}.  \label{itoproces}
\end{equation}%
Here $G$ satisfies%
\begin{equation}
G(t)=\exp -\Phi (t),  \label{differentiaalvergelijkingG}
\end{equation}%
for some primitive of $\Phi $ of $\phi $. We observe that the integration
condition (\ref{integratieconditiemusigmaFG}) may be seen as a first-order
linear differential equation for the auxiliary function $G$ of the
Ito-process (\ref{stochasticdifferentialequationY}), i.e. 
\begin{equation}
\frac{dG}{dt}=\left( \frac{\mu }{\sigma }\frac{\partial \sigma }{\partial Z}+%
\frac{1}{\sigma }\frac{\partial \sigma }{\partial t}+\frac{\sigma }{2}\frac{%
\partial ^{2}\sigma }{\partial Z^{2}}-\frac{\partial \mu }{\partial Z}%
\right) G.  \label{differentiaalG}
\end{equation}

The choice of the function $F$ in (\ref{itoproces}) is free. The two
successive ordinary differential equations leading to $Z(T,Y_{T})$ take now
the form%
\begin{equation}
\left \{ 
\begin{array}{lll}
\frac{dZ}{dt} & = & \mu (t,Z)-\sigma (t,Z)\left( \frac{1}{2}\frac{\partial
\sigma (t,Z)}{\partial Z}+\frac{F(t)}{G(t)}\right) \\ 
Z(0) & = & x_{0}%
\end{array}%
\right.  \label{vergelijkinghmusigmaFG}
\end{equation}%
and%
\begin{equation}
\left \{ 
\begin{array}{lll}
\frac{dZ_{T}}{dx} & = & \frac{\sigma (T,Z_{T})}{G(T)} \\ 
Z_{T}(0) & = & Z_{0}(T),%
\end{array}%
\right.  \label{vergelijkingvmusigmaFG}
\end{equation}%
with $Z(T,Y_{T})=Z_{T}(Y_{T})$. It is of course possible to choose $F\equiv
0 $, and then (\ref{vergelijkinghmusigmaFG}) reduces to (\ref%
{Vergelijkinggewoon1}). However, it may be that a proper adjustment by $F$
makes the ordinary differential equation (\ref{vergelijkinghmusigmaFG})
easier to solve than (\ref{Vergelijkinggewoon1}). This will be illustrated
by the example of the Ornstein-Uhlenbeck process of Example \ref%
{ExampleOrnsteinUhlenbeck} of Section \ref{sectiestochastiek}.

If $G$ satisfies (\ref{differentiaalvergelijkingG}) the formula (\ref%
{integratieconditieFGintroduktie}) expresses the equality $\frac{\partial
^{2}Z(t,Y)}{\partial Y\partial t}=\frac{\partial ^{2}Z(t,Y)}{\partial
t\partial Y}$. Observe that if the function $G$ may be taken equal to $1$,
the integration condition (\ref{integratieconditieFGintroduktie}) reduces to
(\ref{integratieconditiemsintroductie}). Then the Ito-process $Y_{t}$
reduces to Brownian Motion, if $F$ is chosen to be identically $0$.

\section{Existence and uniqueness of solutions of a system of two first
order partial differential equations\label{sectieexistentie}}

Our approach is based on the existence and uniqueness of a solution for the
system of partial differential equations in three variables of the form%
\begin{equation}
\left \{ 
\begin{array}{l}
\frac{\partial Z}{\partial x}=f(x,y,Z) \\ 
\frac{\partial Z}{\partial y}=g(x,y,Z) \\ 
Z(x_{0},y_{0})=z_{0},%
\end{array}%
\right.  \label{systemfirstorderpartiali}
\end{equation}%
with initial condition in one single point $(x_{0},y_{0})\in \mathbb{R}^{2}$%
. If $Z=Z(x,y)$ is a solution of this system of class $C^{2}$, it follows
easily from the equality $\frac{\partial ^{2}Z(x,y)}{\partial y\partial x}=%
\frac{\partial ^{2}Z(x,y)}{\partial x\partial y}$ that$\ f$ and $g$ satisfy%
\begin{equation}
\frac{\partial f}{\partial y}+g\frac{\partial f}{\partial Z}-\frac{\partial g%
}{\partial x}-f\frac{\partial g}{\partial Z}=0.  \label{Integratieconditie}
\end{equation}%
This formula represents the "integration condition" or "compatibility
condition" of the system.

We use the following notations.

\begin{notation}
Let $f$ be a function of two variables $x$ and $y$. With some abuse of
language we may write $f(x,y)=f_{x}(y)$ if $x$ is temporarily fixed\ and $%
f(x,y)=f_{y}(x)$ if $y$ is temporarily fixed. We adopt an analogous
convention for functions of three variables.
\end{notation}

\begin{definition}
Let $z_{0}\in \mathbb{R}$. Let $f:\mathbb{R}^{3}\mathbb{\rightarrow R}$ and $%
g:\mathbb{R}^{3}\mathbb{\rightarrow R}$ be of class $C^{1}$, both uniformly
Lipschitz in the third variable. Consider the system of first-order partial
differential equations%
\begin{equation}
\left \{ 
\begin{array}{l}
\frac{\partial Z}{\partial x}=f(x,y,Z) \\ 
\frac{\partial Z}{\partial y}=g(x,y,Z) \\ 
Z(0,0)=z_{0}.%
\end{array}%
\right.  \label{systemfirstorderpartial}
\end{equation}%
We let $\widetilde{Z}:\mathbb{R}^{2}\rightarrow \mathbb{R}$ be defined by $%
\widetilde{Z}(\overline{x},\overline{y})=\widetilde{Z}_{\overline{x}}(%
\overline{y})$, where $\widetilde{Z}_{\overline{x}}$ satisfies the ordinary
differential equation%
\begin{equation}
\left \{ 
\begin{array}{lll}
\frac{d\widetilde{Z}_{\overline{x}}}{dy} & = & g_{\overline{x}}(y,\widetilde{%
Z}_{\overline{x}}(y)) \\ 
\widetilde{Z}_{\overline{x}}(0) & = & \widetilde{Z}_{0}(\overline{x}),%
\end{array}%
\right.  \label{vergelijkingvertikaal1}
\end{equation}%
with $\widetilde{Z}_{0}$ given by the ordinary differential equation%
\begin{equation}
\left \{ 
\begin{array}{lll}
\frac{d\widetilde{Z}_{0}}{dx} & = & f_{0}(x,\widetilde{Z}_{0}(x)) \\ 
\widetilde{Z}_{0}(0) & = & z_{0}.%
\end{array}%
\right.  \label{vergelijkinghorizontaal1}
\end{equation}
\end{definition}

The following theorem expresses conditions for the existence and uniqueness
of solutions of (\ref{systemfirstorderpartiali}).

\begin{theorem}
\emph{(Exact solution of systems of partial differential equations of first
order)}\label{stellingpartialfg} Let $z_{0}\in \mathbb{R}$. Let $f:\mathbb{R}%
^{3}\mathbb{\rightarrow R}$ be of class $C^{1}$ and $g:\mathbb{R}^{3}\mathbb{%
\rightarrow R}$ be of class $C^{1}$, both uniformly Lipschitz in the third
variable. Assume (\ref{Integratieconditie}) holds. Then $\widetilde{Z}$ is
solution of the system of first-order partial differential equations (\ref%
{systemfirstorderpartial}). As such it is unique and of class $C^{2}$.
\end{theorem}

The Theorem of Frobenius of Differential Geometry \cite{Boothby} implies
local existence and uniqueness of the solution. By the uniform Lipschitz
property the ordinary differential equations (\ref{vergelijkinghorizontaal1}%
) and (\ref{vergelijkingvertikaal1}) have existence and uniqueness of
solutions on any interval. Hence $\widetilde{Z}$ is well-defined and unique
everywhere.

Also, the value of $\widetilde{Z}$ at $(t,x)$ may be obtained by integrating
along any simple continuously differentiable curve, say, $\gamma $ going
from $(0,0)$ to $(t,x)$, i.e. by solving%
\begin{equation*}
\left \{ 
\begin{array}{lll}
\frac{dZ}{d\tau } & = & f((\gamma _{1}(\tau ),\gamma _{2}(\tau ),Z(\tau
))\gamma _{1}^{\prime }(\tau )+g((\gamma _{1}(\tau ),\gamma _{2}(\tau
),Z(\tau ))\gamma _{2}^{\prime }(\tau ) \\ 
Z(0) & = & z_{0}.%
\end{array}%
\right.
\end{equation*}

\section{Functional solutions of stochastic differential equations\label%
{sectiestochastiek}.}

Let $T>0$. To fix ideas, we will always work within an appropriate
probability space $(\Omega ,\mathcal{F},P)$, where $\Omega $ is a
sufficiently rich set, $\mathcal{F=}(\mathcal{F}_{t})_{t\in \lbrack 0,T]}$
is the natural filtration to the Standard Brownian Motion $B_{t}$ on $[0,T]$%
, and $P$ the probability associated to this Standard Brownian Motion. Let $%
\mu ,\sigma :$ $[0,T]\times \mathbb{R\rightarrow R}$ be measurable. For $%
t\in T$, $\omega \in \Omega $ and $x_{0}:\Omega \rightarrow \mathbb{R}$
measurable and of class $L^{2}$ we use often the notation of stochastic
differential equations%
\begin{equation}
\left \{ 
\begin{array}{lllll}
dX_{t} & = & \mu \left( t,X_{t}\right) dt+\sigma \left( t,X_{t}\right) dB_{t}
& \qquad & 0\leq t<T \\ 
X_{0} & = & x_{0} &  & 
\end{array}%
\right.  \label{stochasticdifferentialequation}
\end{equation}%
for the stochastic integral%
\begin{equation*}
X(T,\omega )=x_{0}(\omega )+\int_{0}^{T}\mu \left( t,X(t,\omega \right)
)dt+\int_{0}^{T}\sigma \left( t,X(t,\omega \right) )dB_{t}.
\end{equation*}

We recall that such a stochastic process $X_{t}$ is an \emph{Ito-process}
with respect to $B_{t}$ if it is of the form%
\begin{equation}
X(T,\omega )=x_{0}(\omega )+\int_{0}^{T}F\left( t,\omega \right)
dt+\int_{0}^{T}G\left( t,\omega \right) dB_{t},  \label{formuleitoproces}
\end{equation}%
where $x_{0}$ is $\mathcal{F}_{0}$-measurable, $F$ and $G$ are at any time $%
t $ adapted to $\mathcal{F}_{t}$, and $\int_{0}^{T}\left \vert F\left(
t,\omega \right) \right \vert dt$ and $\int_{0}^{T}\left \vert G\left(
t,\omega \right) \right \vert ^{2}dt$ exist almost surely (later on, for
reasons of simplicity, we will assume that $x_{0}$ is a constant).

For the sake of clarity we recall Ito's Lemma for stochastic processes which
are functions $Z(t,B_{t})$ of time and Brownian Motion and for stochastic
processes which are functions $Z(t,Y_{t})$ of time and a general Ito process 
$Y_{t}$.

\begin{theorem}
\emph{(Ito's Lemma, processes of the form }$Z(t,B_{t})$\emph{) }Let $T>0$
and let $Z:[0,T]\times \mathbb{R}\rightarrow \mathbb{R}$ be of class $C^{12}$%
. The stochastic process $Z(t,B_{t})$ is an Ito-process with respect to $%
B_{t}$ and satisfies for $0\leq t\leq T$ the stochastic differential equation%
\begin{equation*}
\left \{ 
\begin{array}{lll}
dZ_{t} & = & \left( \frac{\partial Z}{\partial t}+\frac{1}{2}\frac{\partial
^{2}Z}{\partial x^{2}}\right) dt+\frac{\partial Z}{\partial x}dB_{t}\text{ \
\  \  \  \  \  \  \  \  \  \  \  \  \  \  \  \  \  \  \  \ }0\leq t<T \\ 
Z_{0} & = & Z(0,x_{0}),%
\end{array}%
\right.
\end{equation*}%
where $x_{0}=B_{0}$.
\end{theorem}

\begin{theorem}
\emph{(Ito's Lemma, processes of the form }$Z(t,Y_{t})$\emph{) }Let $T>0$
and let $Z:[0,T]\times \mathbb{R}\rightarrow \mathbb{R}$ be of class $C^{12}$%
. Let $Y_{t}$ be an Ito-process of the form (\ref{formuleitoproces}), with
initial condition $y_{0}$. The stochastic process $Z(t,Y_{t})$ is an
Ito-process and satisfies for $0\leq t\leq T$ the stochastic differential
equation
\end{theorem}

\begin{equation}
\left \{ 
\begin{array}{lll}
dZ_{t} & = & \left( \frac{\partial Z}{\partial t}+F\frac{\partial Z}{%
\partial x}+\frac{1}{2}G^{2}\frac{\partial ^{2}Z}{\partial x^{2}}\right) dt+G%
\frac{\partial Z}{\partial x}dB_{t} \\ 
Z_{0} & = & Z(0,y_{0}).%
\end{array}%
\right.  \label{stochasticdifferentialequationItoFG}
\end{equation}

The Main Theorem on the existence of global solutions of stochastic
differential equations in the form of deterministic functions of Ito
processes is as follows.

\begin{theorem}
\label{TheoremequivalenceFG}\emph{(Main Theorem) }Let $T>0$ and $x_{0}\in 
\mathbb{R}$. Let $\mu :[0,T]\times \mathbb{R\rightarrow R}$ be of class $%
C^{1}$, and $\sigma :[0,T]\times \mathbb{R\rightarrow R}^{+}\backslash \{0\}$
be of class $C^{2}$, both uniformly Lipschitz in the second variable, and
with $\frac{\partial \sigma }{\partial X}$ bounded. Assume 
\begin{equation}
\frac{\partial \mu }{\partial X}-\frac{\mu }{\sigma }\frac{\partial \sigma }{%
\partial X}-\frac{1}{\sigma }\frac{\partial \sigma }{\partial t}-\frac{%
\sigma }{2}\frac{\partial ^{2}\sigma }{\partial ^{2}X}=\phi (t)
\label{integratieconditiemusigmaFG}
\end{equation}%
for some real continuous function $\phi $. Let $\Phi (t)=\int_{0}^{T}\phi
(t)dt$ and $G(t)=\exp -\Phi (t)$. Let $F$ be a real function of class $C^{1}$
and $Y_{t}$ be the Ito process given by 
\begin{equation}
Y_{t}=y_{0}+\int_{0}^{t}F\left( s\right) ds+\int_{0}^{t}G\left( s\right)
dB_{s},  \label{stochasticdifferentialequationY}
\end{equation}%
where $y_{0}$ is some constant. Then there exists a unique function $%
Z:[0,T]\times \mathbb{R\rightarrow R}$ of class $C^{23}$ such that $%
Z(t,Y_{t})$ is an Ito-process with respect to $B_{t}$ satisfying (\ref%
{stochasticdifferentialequation}). In fact, for all $(\overline{t},\overline{%
x})\in \lbrack 0,T]\times \mathbb{R}$ the value $Z(\overline{t},\overline{x}%
) $ may be determined by solving successively the ordinary differential
equations%
\begin{equation}
\left \{ 
\begin{array}{lll}
\frac{d\widetilde{Z}}{dt} & = & \mu (t,\widetilde{Z})-\sigma (t,\widetilde{Z}%
)\left( \frac{1}{2}\frac{\partial \sigma (t,\widetilde{Z})}{\partial Z}+%
\frac{F(t)}{G(t)}\right) \\ 
\widetilde{Z}(0) & = & x_{0}%
\end{array}%
\right.  \label{vergelijkinghorizontaalmusigmaFG}
\end{equation}%
and%
\begin{equation}
\left \{ 
\begin{array}{lll}
\frac{d\widetilde{Z}_{\overline{t}}}{dx} & = & \frac{\sigma (\overline{t},%
\widetilde{Z}_{\overline{t}})}{G(\overline{t})} \\ 
\widetilde{Z}_{\overline{t}}(0) & = & \widetilde{Z}(\overline{t}),%
\end{array}%
\right.  \label{vergelijkingvertikaalmusigmaFG}
\end{equation}%
with $Z(\overline{t},\overline{x})=\widetilde{Z}_{\overline{t}}(\overline{x}%
) $.\newline
Conversely, if (\ref{stochasticdifferentialequation}) has a solution of
class $C^{23}$ of the form $Z(t,Y_{t})$, where $Y_{t}$ is given by (\ref%
{stochasticdifferentialequationY}), with $F$ and $G\neq 0$ of class $C^{1}$,
formula (\ref{integratieconditiemusigmaFG}) holds with $\phi (t)=-G^{\prime
}(t)/G(t)$.
\end{theorem}

Observe that the Main Theorem expresses path-independence of the process $%
X=Z $ with respect to the process $Y$: if $\omega ,\omega ^{\prime }\in
\Omega $ are such that $Y_{t}(\omega )=Y_{t}(\omega ^{\prime })$ at some
time $t$, it holds that $X_{t}(\omega )=Z(t,Y_{t}(\omega ))=Z(t,Y_{t}(\omega
^{\prime }))=X_{t}(\omega ^{\prime })$. If $G^{\prime }\neq 0$ the process
does not have path-independence with respect to Brownian motion. Functional
dependence $Z(t,B_{t})$ on time and Brownian Motion is characterized by the
following corollary, with a simpler integration condition and simpler
ordinary differential equations for the function $Z$.

\begin{theorem}
\label{Theoremequivalence)}Let $T>0$ and $x_{0}\in \mathbb{R}$. Let $\mu
:[0,T]\times \mathbb{R\rightarrow R}$ be of class $C^{1}$, and $\sigma
:[0,T]\times \mathbb{R\rightarrow R}$ be of class $C^{12}$, both uniformly
Lipschitz in the second variable, and with $\frac{\partial \sigma }{\partial
X}$ bounded. Consider the stochastic differential equation (\ref%
{stochasticdifferentialequation}). Assume 
\begin{equation}
\sigma \frac{\partial \mu }{\partial X}-\mu \frac{\partial \sigma }{\partial
X}-\frac{\partial \sigma }{\partial t}-\frac{\sigma ^{2}}{2}\frac{\partial
^{2}\sigma }{\partial X^{2}}=0.  \label{partialdifferentialequationmusigma}
\end{equation}%
Then there exists a unique function $Z:[0,T]\times \mathbb{R\rightarrow R}$
of class $C^{23}$ such that $Z(t,B_{t})$ is an Ito-process with respect to $%
B_{t}$ satisfying (\ref{stochasticdifferentialequation}). In fact, for all $(%
\overline{t},\overline{x})\in \lbrack 0,T]\times \mathbb{R}$ the value $Z(%
\overline{t},\overline{x})$ may be determined by solving successively the
ordinary differential equations%
\begin{equation}
\left \{ 
\begin{array}{lll}
\frac{d\widetilde{Z}}{dt} & = & \mu (t,\widetilde{Z})-\frac{1}{2}\sigma (t,%
\widetilde{Z})\frac{\partial \sigma (t,\widetilde{Z})}{\partial Z} \\ 
\widetilde{Z}(0) & = & x_{0}%
\end{array}%
\right.  \label{vergelijkinghorizontaalmusigma}
\end{equation}%
and%
\begin{equation}
\left \{ 
\begin{array}{lll}
\frac{d\widetilde{Z}_{\overline{t}}}{dx} & = & \sigma (\overline{t},%
\widetilde{Z}_{\overline{t}}) \\ 
\widetilde{Z}_{\overline{t}}(0) & = & \widetilde{Z}(\overline{t}),%
\end{array}%
\right.  \label{vergelijkingvertikaalmusigma}
\end{equation}%
with $Z(\overline{t},\overline{x})=\widetilde{Z}_{\overline{t}}(\overline{x}%
) $.\newline
Conversely, if (\ref{stochasticdifferentialequation}) has a solution of
class $C^{23}$ of the form $Z(t,B_{t})$, formula (\ref%
{partialdifferentialequationmusigma}) holds.
\end{theorem}

\begin{proof}[Proof of the Main Theorem]
Let $f$ $:[0,T]\times \mathbb{R\times R\rightarrow R}$ and $g$ $:[0,T]\times 
\mathbb{R\times R\rightarrow R}$ be defined by 
\begin{equation}
\left \{ 
\begin{array}{lll}
f(t,x,Z) & = & \mu (t,Z)-\frac{1}{2}\sigma (t,Z)\frac{\partial \sigma (t,Z)}{%
\partial Z}-\frac{F(t)}{G(t)}\sigma (t,Z) \\ 
g(t,x,Z) & = & \frac{\sigma (t,Z)}{G(t)}.%
\end{array}%
\right.  \label{definitiefg}
\end{equation}%
Then $f$ is of class $C^{1\infty 1}$ and $g$ is of class $C^{1\infty 2}$ and
both are uniformly Lipschitz in the third variable. Consider the system of
two partial differential equations%
\begin{equation}
\left \{ 
\begin{array}{lll}
\frac{\partial Z}{\partial t} & = & f(t,x,Z) \\ 
\frac{\partial Z}{\partial x} & = & g(t,x,Z) \\ 
Z(0,0) & = & x_{0}.%
\end{array}%
\right.  \label{systeemZfgFG}
\end{equation}%
The conditions for the integration of the system (\ref{systeemZfgFG}) are
satisfied, since%
\begin{eqnarray*}
&&\frac{\partial f}{\partial x}+g\frac{\partial f}{\partial Z}-\frac{%
\partial g}{\partial t}-f\frac{\partial g}{\partial Z} \\
&=&\frac{\sigma }{G}\frac{\partial \left( \mu -\frac{1}{2}\sigma \frac{%
\partial \sigma }{\partial Z}-\frac{F}{G}\sigma )\right) }{\partial Z}-\frac{%
\partial (\sigma /G)}{\partial t}-\left( \mu -\frac{1}{2}\sigma \frac{%
\partial \sigma }{\partial Z}-\frac{F}{G}\sigma \right) \frac{\partial
(\sigma /G)}{\partial Z} \\
&=&\frac{1}{G}\left( \sigma \frac{\partial \mu }{\partial Z}-\mu \frac{%
\partial \sigma }{\partial Z}-\frac{\partial \sigma }{\partial t}-\frac{%
\sigma ^{2}}{2}\frac{\partial ^{2}\sigma }{\partial Z^{2}}+\frac{G^{\prime }%
}{G}\sigma \right) \\
&=&0.
\end{eqnarray*}%
By Theorem \ref{stellingpartialfg} the system (\ref{systeemZfgFG}) has a
solution $Z:[0,T]\times \mathbb{R\rightarrow R}$ at least of class $C^{2}$.
It follows from the identities%
\begin{equation*}
\frac{\partial ^{2}Z}{\partial x^{2}}=\frac{\sigma (t,Z(t,x))}{G^{2}(t)}%
\frac{\partial \sigma (t,Z(t,x))}{\partial Z}
\end{equation*}%
and 
\begin{equation*}
\frac{\partial ^{3}Z}{\partial x^{3}}=\frac{\sigma (t,Z(t,x))}{G^{3}(t)}%
\left( \frac{\partial \sigma (t,Z(t,x))}{\partial Z}\right) ^{2}+\frac{%
\sigma ^{2}(t,Z(t,x))}{G^{3}(t)}\frac{\partial ^{2}\sigma (t,Z(t,x))}{%
\partial Z^{2}}
\end{equation*}%
that the solution $Z$ is in fact of class $C^{23}$. One verifies that $Z$
satisfies 
\begin{equation}
\left \{ 
\begin{array}{lll}
\frac{\partial Z}{\partial t}+F\frac{\partial Z}{\partial x}+\frac{1}{2}G^{2}%
\frac{\partial ^{2}Z}{\partial x^{2}} & = & \mu (t,Z) \\ 
\frac{\partial Z}{\partial x} & = & \frac{\sigma (t,Z)}{G(t)} \\ 
Z(0,0) & = & x_{0},%
\end{array}%
\right.  \label{systeem2ZfgFG}
\end{equation}%
and then by Ito's Lemma the Ito-process $Z_{t}\equiv Z(t,Y_{t})$ satisfies
the stochastic differential equation%
\begin{equation*}
\left \{ 
\begin{array}{lll}
dZ_{t} & = & \left( \frac{\partial Z}{\partial t}+F\frac{\partial Z}{%
\partial x}+\frac{1}{2}G^{2}\frac{\partial ^{2}Z}{\partial x^{2}}\right) dt+G%
\frac{\partial Z}{\partial x}dB_{t} \\ 
Z_{0} & = & x_{0}.%
\end{array}%
\right.
\end{equation*}%
By (\ref{systeem2ZfgFG}) it satisfies also the stochastic differential
equation (\ref{stochasticdifferentialequation}).

As for uniqueness, we observe first that by Theorem \ref{stellingpartialfg}
the function $Z$ is the unique solution of class $C^{2}$ (in fact of class $%
C^{23}$) of the system (\ref{systeemZfgFG}). For fixed $t>0$ let $\zeta
:[0,T]\times \mathbb{R\rightarrow R}$ be of class $C^{23}$ such that $\zeta
(t,Y_{t})$ is an Ito-process satisfying (\ref{stochasticdifferentialequation}%
). By the Existence-Uniqueness Theorem for stochastic differential equations 
\cite{Oksendahl}, almost surely%
\begin{equation*}
\sup_{0\leq t\leq T}\left \vert Z_{t}-\zeta _{t}\right \vert =0.
\end{equation*}%
Hence $Z(t,Y_{t}(\omega ))=\zeta (t,Y_{t}(\omega ))$ almost surely. For $t>0$
the range of the stochastic variable $B_{t}$ is the whole of $\mathbb{R}$,
so because the positive and continuous function $G$ is has a non-zero lower
bound on $[0,t]$, the range of the stochastic variable $\int_{0}^{t}G\left(
s\right) dB_{s}$ is also the whole of $\mathbb{R}$, hence because $F$ is
bounded on $[0,t]$ the range of the stochastic variable $Y_{t}=y_{0}+%
\int_{0}^{t}F\left( s\right) ds+\int_{0}^{t}G\left( s\right) dB_{s}$ is also
the whole of $\mathbb{R}$. Hence $Z(t,x)=\zeta (t,x)$ almost surely for $%
x\in \mathbb{R}$ with respect to the measure on $\mathbb{R}$ induced by $%
\mathcal{F}$ and $P$. By continuity of $Z$ and $\zeta $ we have $%
Z(t,x)=\zeta (t,x)$ for all $(t,x)\in (0,T]\times \mathbb{R}$, and since $%
Z(0,x)=\zeta (0,x)=x_{0}$, in fact for all $(t,x)\in \lbrack 0,T]\times 
\mathbb{R}$. Hence $Z=\zeta $.

The converse follows from the equality $\frac{\partial ^{2}Z(t,Y)}{\partial
Y\partial t}=\frac{\partial ^{2}Z(t,Y)}{\partial t\partial Y}$.
\end{proof}

\begin{theorem}
\label{stelling distributiefunctie}Assume the conditions of Theorem \ref%
{TheoremequivalenceFG} are satisfied. Let $Y_{t}$ be an Ito-process such
that $Z(t,Y_{t})$ solves the stochastic differential equation (\ref%
{stochasticdifferentialequation}) Let $0<t\leq T$ and $D$ be the cumulative
distribution function of $Y_{t}$. Then for all $x\in \mathbb{R}$%
\begin{equation*}
\Pr \left \{ X_{t}\leq x\right \} =D(Z_{t}^{-1}(x)).
\end{equation*}%
In particular, if $Y_{t}=B_{t}$%
\begin{equation*}
\Pr \left \{ X_{t}\leq x\right \} =\mathcal{N}(Z_{t}^{-1}(x)).
\end{equation*}
\end{theorem}

The proof is obvious, noting that $Z^{-1}(t,Y_{t})$ is well-defined for
fixed $t$, since $\frac{\partial Z}{\partial x}=\frac{\sigma (t,Z)}{G(t)}$
is always positive.

\bigskip

\textbf{Remarks.}

\begin{enumerate}
\item In \cite{vandenberglobry} the problem of path-independence was studied
in a discrete setting from an asymptotic point-of-view. Roughly spoken,
solutions of stochastic difference equations 
\begin{equation}
\delta X_{t}=\mu (t,X_{t})\delta t+\sigma (t,X_{t})\delta W_{t},
\label{stochasticdifferenceequation}
\end{equation}%
where $\delta W_{t}=\pm \sqrt{\delta t}$ is the Wiener Walk and $\delta
t\rightarrow 0$, happen to have in the limit the same probability
distribution as a deterministic function $Z(t,W_{t})$ - i.e. some deformed
Normal Distribution like in Theorem \ref{stelling distributiefunctie} -
provided (\ref{partialdifferentialequationmusigma}) holds. The condition (%
\ref{partialdifferentialequationmusigma}) expresses a form of near
path-dependence on microscopic level. Observe that an upward movement $%
\delta W_{t}=+\sqrt{\delta t}$ followed by a downward movement $\delta
W_{t}=-\sqrt{\delta t}$ yields the same value as a downward movement $\delta
W_{t}=-\sqrt{\delta t}$ followed by an upward movement $\delta W_{t}=+\sqrt{%
\delta t}$. This is not true for a general process $X_{t}$ given by (\ref%
{stochasticdifferenceequation}), but if (\ref%
{partialdifferentialequationmusigma}) holds the values of an upward movement
followed by a downward movement and a downward movement followed by an
upward movement happen to be sufficiently close to permit the above limit
property for its probability distribution. The property follows by applying
Taylor-expansions to the increments $\delta X_{t}$.

\item The second-order integration condition (\ref%
{integratieconditiemusigmaFG}) may be simplified and also be solved for $\mu 
$. Firstly, put%
\begin{equation}
\nu (t,X)=\mu (t,X)-\frac{1}{2}\sigma (t,X)\frac{\partial \sigma (t,X)}{%
\partial X}.  \label{integratieconditiestratonovich}
\end{equation}%
Then (\ref{partialdifferentialequationmusigma}) becomes the first-order
linear partial differential equation%
\begin{equation}
\sigma \frac{\partial \nu }{\partial X}-\frac{\partial \sigma }{\partial X}%
\nu =\frac{\partial \sigma }{\partial t}+\phi (t).
\label{Integratieconditiemuessigma}
\end{equation}%
When solved for $\nu $, with $\mu (t,X)=\nu (t,X)+\frac{1}{2}\sigma (t,X)%
\frac{\partial \sigma (t,X)}{\partial X}$ one finds%
\begin{equation}
\mu (t,X)=\left( \frac{1}{2}\frac{\partial \sigma (t,X)}{\partial X}%
+\int \limits_{0}^{X}\frac{\frac{\partial \sigma (t,\xi )}{\partial t}+\phi
(t)}{\sigma ^{2}(t,\xi )}d\xi +\gamma (t)\right) \sigma (t,X),
\label{Integratieconditiemu}
\end{equation}%
for some function $\gamma $ of class $C^{1}$.

\item Some special cases lead to simplifications of the integration
condition (\ref{Integratieconditiemu}). For (\ref%
{stochasticdifferentialequation}) to have solutions of the form $Z(t,B_{t})$
one has $\phi =0$ and (\ref{Integratieconditiemu}) reduces to 
\begin{equation}
\mu (t,X)=\left( \frac{1}{2}\frac{\partial \sigma (t,X)}{\partial X}%
+\int \limits_{0}^{X}\frac{\frac{\partial \sigma (t,\xi )}{\partial t}}{%
\sigma ^{2}(t,\xi )}d\xi +\gamma (t)\right) \sigma (t,X).
\label{Integratieconditiemupadonafhankelijk}
\end{equation}%
In the autonomous case one also has $\phi =0$ and eliminating all dependence
on $t$ in (\ref{Integratieconditiemupadonafhankelijk}) one finds%
\begin{equation}
\mu (X)=\left( \frac{1}{2}\sigma ^{\prime }(X)+c\right) \sigma (X)
\label{Integratieconditiemuautonoom}
\end{equation}%
for some constant $c$. Moreover, if $\sigma $ is linear, $\mu $ is must also
be linear.
\end{enumerate}

\textbf{Examples.}

\begin{enumerate}
\item \textbf{Autonomous case.} Consider the stochastic differential equation%
\begin{equation*}
\left \{ 
\begin{array}{lllll}
dX_{t} & = & \mu (X_{t})dt+\sigma (X_{t})dB_{t} & \qquad & 0\leq t<T \\ 
X_{0} & = & x_{0}, &  & 
\end{array}%
\right.
\end{equation*}%
with $x_{0}\in \mathbb{R}$, and $\mu $ of class $C^{1}$ and $\sigma \neq 0$
of class $C^{2}$. We saw that for functional solutions of the form $%
Z(t,Y_{t})$, where $Y_{t}$ is an Ito Process of the form (\ref%
{stochasticdifferentialequationY}), in fact $Y_{t}$ is equal to $B_{t}$, and
that the trend $\mu $ must satisfy (\ref{Integratieconditiemuautonoom}). If
the remaining conditions of Theorem \ref{Theoremequivalence)} are satisfied
the function $Z$ may be determined by means of the equations (\ref%
{vergelijkinghorizontaalmusigma}) and (\ref{vergelijkingvertikaalmusigma}),
which become the differential equations with separable variables%
\begin{equation}
\left \{ 
\begin{array}{lll}
\frac{d\widetilde{Z}}{dt} & = & c\sigma (\widetilde{Z}) \\ 
\widetilde{Z}(0) & = & x_{0}%
\end{array}%
\right.  \label{vergelijkinghorizontaalautonoom}
\end{equation}%
and%
\begin{equation}
\left \{ 
\begin{array}{lll}
\frac{d\widetilde{Z}_{\overline{t}}}{dx} & = & \sigma (\widetilde{Z}_{%
\overline{t}}) \\ 
\widetilde{Z}_{\overline{t}}(0) & = & \widetilde{Z}(\overline{t}).%
\end{array}%
\right.  \label{vergelijkingvertikaalautonoom}
\end{equation}%
A well-known special case is the Geometric Brownian Motion. It satisfies the
stochastic differential equation%
\begin{equation*}
\left \{ 
\begin{array}{lllll}
dX_{t} & = & \widehat{\mu }X_{t}dt+\widehat{\sigma }X_{t}dB_{t} & \qquad & 
0\leq t<T \\ 
X_{0} & = & x_{0}, &  & 
\end{array}%
\right.
\end{equation*}%
with $\widehat{\mu }\in \mathbb{R}$ and $\widehat{\sigma }>0$. One has $c=(%
\widehat{\mu }-\widehat{\sigma }^{2}/2)/\widehat{\sigma }$, hence the
solution of (\ref{vergelijkinghorizontaalautonoom}) is $Z(\overline{t}%
)=x_{0}\exp (\widehat{\mu }-\widehat{\sigma }^{2}/2)\overline{t}$ (this
formula is perhaps most easily found applying (\ref%
{vergelijkinghorizontaalmusigma}) directly). Then the solution of (\ref%
{vergelijkingvertikaalautonoom}) is $Z_{\overline{t}}(\overline{x}%
)=x_{0}\exp \left( (\widehat{\mu }-\widehat{\sigma }^{2}/2)\overline{t}+%
\widehat{\sigma }\overline{x}\right) $ and one derives the well-known
formula $Z(t,B_{t})$\linebreak $=x_{0}\exp \left( (\widehat{\mu }-\widehat{%
\sigma }^{2}/2)t+\widehat{\sigma }B_{t}\right) $.

\item \label{ExampleOrnsteinUhlenbeck}\textbf{Ornstein-Uhlenbeck process. }%
This process is given by the stochastic differential equation%
\begin{equation*}
\left \{ 
\begin{array}{lllll}
dR_{t} & = & \theta (\widehat{\mu }-R_{t})dt+\widehat{\sigma }dB_{t} & \qquad
& 0\leq t<T \\ 
R_{0} & = & r_{0}, &  & 
\end{array}%
\right.
\end{equation*}%
where $\theta ,\widehat{\mu },\widehat{\sigma }\neq 0$ and $r_{0}$ are all
constants. With $\mu (t,R_{t})=\theta (\widehat{\mu }-R_{t})$ and $\sigma
(t,R_{t})=\widehat{\sigma }$ the integration condition (\ref%
{integratieconditiemusigmaFG}) becomes%
\begin{equation*}
\frac{\partial \mu }{\partial R}-\frac{\mu }{\sigma }\frac{\partial \sigma }{%
\partial R}-\frac{1}{\sigma }\frac{\partial \sigma }{\partial t}-\frac{%
\sigma }{2}\frac{\partial ^{2}\sigma }{\partial ^{2}R}=-\theta .
\end{equation*}%
Hence $G(t)=ce^{\theta t}$, for some $c\in \mathbb{R}$. By Theorem \ref%
{TheoremequivalenceFG} one has $R_{t}=Z(t,Y_{t})$, where $Y_{t}$ is of the
form 
\begin{equation*}
Y_{t}=y_{0}+\int_{0}^{t}F\left( s\right) ds+c\int_{0}^{t}e^{\theta s}dB_{s};
\end{equation*}%
here $y_{0}\in \mathbb{R}$, $F$ is of class $C^{1}$ on $[0,T]$, and $Z$ is
of class $C^{23}$ and is determined by solving successively the auxiliary
differential equations%
\begin{equation}
\left \{ 
\begin{array}{lll}
\frac{d\widetilde{Z}}{ds} & = & \theta (\widehat{\mu }-\widetilde{Z})-%
\widehat{\sigma }\left( 0+\frac{F(s)}{ce^{\theta s}}\right) =\theta \widehat{%
\mu }-\frac{\widehat{\sigma }}{c}F(s)e^{-\theta s}-\theta \widetilde{Z} \\ 
\widetilde{Z}(0) & = & r_{0}%
\end{array}%
\right.  \label{OU1}
\end{equation}%
and%
\begin{equation}
\left \{ 
\begin{array}{lll}
\frac{d\widetilde{Z}_{t}}{dx} & = & \frac{\widehat{\sigma }}{ce^{\theta t}}=%
\frac{\widehat{\sigma }}{c}e^{-\theta t} \\ 
\widetilde{Z}_{t}(0) & = & \widetilde{Z}(t).%
\end{array}%
\right.  \label{OU2}
\end{equation}%
We may choose $y_{0},c$ and $F$ freely in order to obtain simplifications
and assume that $y_{0}=0$, $c=1$ and $F(s)=\frac{\theta \widehat{\mu }}{%
\widehat{\sigma }}e^{\theta s}$. Then (\ref{OU1}) becomes $\frac{d\widetilde{%
Z}}{ds}=-\theta \widetilde{Z}$, with $\widetilde{Z}(0)=r_{0}$. Hence $%
\widetilde{Z}(t)=r_{0}e^{-\theta t}$. Solving (\ref{OU2}) we find%
\begin{equation*}
Z(t,Y_{t})=r_{0}e^{-\theta t}+\widehat{\sigma }e^{-\theta t}Y_{t}.
\end{equation*}%
With $Y_{t}=\frac{\theta \widehat{\mu }}{\overline{\sigma }}%
\int_{0}^{t}e^{\theta s}ds+\int_{0}^{t}e^{\theta s}dB_{s}=\frac{\widehat{\mu 
}}{\overline{\sigma }}(e^{\theta t}-1)+\int_{0}^{t}e^{\theta s}dB_{s}$ we
derive the well-known formula%
\begin{equation*}
R_{t}=r_{0}e^{-\theta t}+\widehat{\mu }(1-e^{-\theta t})+\widehat{\sigma }%
\int_{0}^{t}e^{\theta (s-t)}dB_{s}\text{.}
\end{equation*}

\item \textbf{Homogeneous linear stochastic differential equations}.
Consider the stochastic differential equation%
\begin{equation}
\left \{ 
\begin{array}{lllll}
dX_{t} & = & \alpha (t)X_{t}dt+\beta (t)X_{t}dB_{t} & \qquad & 0\leq t<T \\ 
X_{0} & = & x_{0}. &  & 
\end{array}%
\right.  \label{stochasticdifferentialequationlinear}
\end{equation}%
We consider the non-degenerate case where $\beta (t)\neq 0$. It is
well-known \cite{Arnold} that the solution of (\ref%
{stochasticdifferentialequationlinear}) is given by%
\begin{equation}
X_{t}=\exp \left( \int \limits_{0}^{t}\alpha (s)-\frac{1}{2}\beta
^{2}(s)ds+\int \limits_{0}^{t}\beta (s)dB_{s}\right) .
\label{Oplossinghomogeenlineair}
\end{equation}%
We give a direct derivation, not using verification by Ito's lemma. The
integration condition (\ref{differentiaalG}) becomes%
\begin{equation}
\frac{G^{\prime }(t)}{G(t)}=\frac{\beta ^{\prime }(t)}{\beta (t)}
\label{integratieconditielineair}
\end{equation}%
Hence $G(t)=c\beta (t)$ for some $c\in \mathbb{R}$. One may put $c=1$ and $%
F=0$. Then (\ref{vergelijkinghorizontaalmusigmaFG}) and (\ref%
{vergelijkingvertikaalmusigmaFG}) become%
\begin{equation}
\left \{ 
\begin{array}{lll}
\frac{d\widetilde{Z}}{dt} & = & (\alpha (t)-\frac{1}{2}\beta ^{2}(t))%
\widetilde{Z} \\ 
\widetilde{Z}(0) & = & x_{0}%
\end{array}%
\right.  \label{lineairhorizontaal}
\end{equation}%
and%
\begin{equation}
\left \{ 
\begin{array}{lll}
\frac{d\widetilde{Z}_{\overline{t}}}{dx} & = & \widetilde{Z}_{\overline{t}}
\\ 
\widetilde{Z}_{\overline{t}}(0) & = & \widetilde{Z}(\overline{t}).%
\end{array}%
\right.  \label{lineairvertikaal}
\end{equation}%
The solution of (\ref{stochasticdifferentialequationlinear}) is $%
X_{t}=Z(t,Y_{t})$, with 
\begin{equation*}
Y_{t}=x_{0}=\int \limits_{0}^{t}\beta (s)dB_{s},
\end{equation*}%
Solving the equations (\ref{lineairhorizontaal}) and (\ref{lineairvertikaal}%
), we derive (\ref{Oplossinghomogeenlineair}).

\item \textbf{A non-autonomous nonlinear stochastic differential equation}.
Consider the stochastic differential equation (\ref%
{stochasticdifferentialequation}) with 
\begin{equation*}
\left \{ 
\begin{array}{l}
\mu (t,X_{t})=\frac{\exp \left( -\frac{X_{t}^{2}}{2}\right) }{t+1}%
\int_{0}^{X_{t}}\exp \frac{\xi ^{2}}{2}d\xi -\frac{X_{t}(t+1)^{2}}{2}\exp
\left( -X_{t}^{2}\right) +\exp \left( -\frac{X_{t}^{2}}{2}\right) \\ 
\sigma (t,X_{t})=(t+1)\exp \left( -\frac{X_{t}^{2}}{2}\right) .%
\end{array}%
\right.
\end{equation*}
\end{enumerate}

One verifies that $\mu $ and $\sigma $ satisfy the integration condition (%
\ref{partialdifferentialequationmusigma}). On any time-interval $[0,T],T>0$
the functions $\sigma $ and $\frac{\partial \sigma }{\partial X}$ are
uniformly bounded and $\mu (t,X)$ has an upper bound of the form $%
K\left \vert X\right \vert +L$, where the constants $K$ and $L$ do not depend
on $t$. This means that the remaining conditions of Theorem \ref%
{Theoremequivalence)} are also satisfied. Hence (\ref%
{stochasticdifferentialequation}) has a solution $Z(t,B_{t})$, which may be
found by solving (in principle) the ordinary differential equations (\ref%
{vergelijkinghorizontaalmusigma}), i.e.%
\begin{equation*}
\left \{ 
\begin{array}{lll}
\frac{d\widetilde{Z}}{dt} & = & \frac{1}{t+1}\exp \left( -\frac{\widetilde{Z}%
^{2}}{2}\right) \int_{0}^{\widetilde{Z}}\exp \frac{\xi ^{2}}{2}d\xi +\exp
\left( -\frac{\widetilde{Z}^{2}}{2}\right) \\ 
\widetilde{Z}(0) & = & x_{0},%
\end{array}%
\right.
\end{equation*}%
and (\ref{vergelijkingvertikaalmusigma}), i.e.%
\begin{equation*}
\left \{ 
\begin{array}{lll}
\frac{d\widetilde{Z}_{\overline{t}}}{dx} & = & (\overline{t}+1)\exp \left( -%
\frac{\widetilde{Z}_{\overline{t}}}{2}\right) \\ 
\widetilde{Z}_{\overline{t}}(0) & = & \widetilde{Z}_{0}(\overline{t}).%
\end{array}%
\right.
\end{equation*}

\textbf{Acknowledgement}. I thank T.Sari (University of Mulhouse/INRA
Montpellier) for some enlightening discussions on differential geometry.


\begin{thebibliography}{9}
\bibitem{Arnold} L. Arnold, \textit{Stochastic Differential Equations:
Theory and Applications}, Wiley (1974).

\bibitem{vandenberglobry} I.P. van den Berg, E. Amaro, \textit{Nearly
recombining processes and the calculation of expectations}, ARIMA, Vol. 9,
p. 389-417 (2008).

\bibitem{Boothby} W.M. Boothby, \textit{An introduction to differentiable
manifolds and Riemannian geometry}, 2$^{nd}$ ed., Pure and Applied
Mathematics, 120. Academic Press, Inc., Orlando (1986).

\bibitem{CourantHilbert} R. Courant, D. Hilbert, \textit{Methods of
mathematical physics. Vol. II. Partial differential equations}, Wiley \&
Sons, Inc., New York, 1989 (reprint).

\bibitem{Doss} H. Doss, \textit{Liens entre \'{e}quations diff\'{e}%
rentielles stochastiques et ordinaires}, Ann. Inst. Henri Poincar\'{e}, Vol
XIII, n$^{\circ }$2, p. 99-125 (1977).

\bibitem{Oksendahl} B. \O ksendal, \textit{Stochastic differential
equations, an introduction with applications}, Universitext,
Springer-Verlag, Berlin (2003), 6$^{\text{th}}$ edition.

\bibitem{Protter} Ph. Protter, \textit{Stochastic integration and
differential equations. A new approach}, Springer-Verlag, Berlin (1990).

\bibitem{Zwillinger} D. Zwillinger, \textit{Handbook of differential
equations}, 2$^{nd}$ ed. Academic Press, Boston (1992).
\end{thebibliography}
\end{document}